%% file: frcont1.tex
\documentclass[twoside,reqno,10pt]{amsart}

 \usepackage{amsmath}
 \usepackage{url}
\include{command11}

 \usepackage{graphicx,color, amssymb}
 
 \usepackage{lineno}
\begin{document}
  
   \title[Continuity of fractional velocity]{On the conditions for existence and continuity of fractional velocity \thanks{The work has been supported in part by a grant from Research Fund - Flanders (FWO), contract number 0880.212.840.}}
   
   \author {Dimiter Prodanov}
   \address{Correspondence: Environment, Health and Safety, IMEC vzw, Kapeldreef 75, 3001 Leuven, Belgium
   e-mail: Dimiter.Prodanov@imec.be, dimiterpp@gmail.com}

   %%%%%%%%%%%%%%%%%%%%%%5
 
   	\begin{abstract}
	 H\"older functions represent mathematical models of nonlinear physical phenomena. This work investigates the general conditions of existence of fractional velocity as a localized generalization of ordinary derivative with regard to the exponent order.
	 Fractional velocity is defined as the limit of the difference quotient of the function's increment  and the difference of its argument raised to a fractional power.
	 A relationship to the point-wise  H\"older exponent of a function, its point-wise oscillation and the  existence of fractional velocity is established. 
	 It is demonstrated that wherever the fractional velocity of non-integral order is continuous then it vanishes. 
	 The work further demonstrates the use of fractional velocity as a tool for characterization of the discontinuity set of the derivatives of functions thus providing a natural characterization of strongly non-linear local behavior.
	 Finally the equivalence with the Kolwankar-Gangal local fractional derivative is investigated. 
  
   	 \medskip
   		
   		
 \smallskip
   		

 \end{abstract}

	   	\maketitle  
	   	
   	%\linenumbers
   	
   	%% main text
   	%%%%%%%%%%%%%%%%%%%%%
   	%	Introduction
   	%%%%%%%%%%%%%%%%%%%%%
   	\section{Introduction}
   	\label{seq:intro}
   	%%%%%%%%%%%%%%%%%%%%%
 	
   	Derivatives can be viewed as mathematical idealizations of the linear growth. 
   	On the other hand, mathematical descriptions of strongly non-linear phenomena necessitate certain relaxation of the linearity assumption. While this can be achieved in many ways, the present work focuses entirely on local descriptions in terms of limits of difference quotients.  	
   	Difference quotients of functions of fractional order have been considered initially by du Bois-Reymond  \cite{BoisReymond1875} and later by Faber \cite{Faber1909} in their studies of the point-wise differentiability of functions. The concept implied what is now known as Holder-continuity of the function.
   	While these initial development followed from purely mathematical interest, later works were inspired from physical research questions. Cherbit \cite{Cherbit1991} and later on Ben Adda and Cresson  \cite{Adda2001} introduced the notion of \textsl{fractional velocity} as the limit of the fractional difference quotient. Their main application was the study of fractal phenomena and physical processes for which the instantaneous velocity was not well defined \cite{Cherbit1991}.  	
   	%Notably, research in scaling of physical systems and  systems exhibiting fractal behavior were modeled by local (integral) fractional derivatives \cite{Kolwankar1998}.  
 
    Existence of the fractional velocity was demonstrated for some classes of functions in \cite{Prodanov2015, Prodanov2016}.
 	%The present work builds further on the use of \textsl{fractional variation} operators \cite{Prodanov2015}, which in limit are equivalent to the \textsl{fractional velocity} introduced by Cherbit \cite{Cherbit1991}.
 	This work establishes further the general conditions of existence of fractional velocity using the formalism of \textsl{fractional variation} operators \cite{Prodanov2015}.
 	The  most fundamental result of the present work is that for fractional orders  fractional velocity is continuous only if it is zero.	
	The  set of discontinuities of fractional velocity is  characterized and used to describe the local change of the function in terms of the fractional Taylor-Lagrange property.
	%singular points of the ordinary derivative. 
	Finally,  the relationship between fractional velocities and the localized versions of fractional derivatives in the sense of Kolwankar-Gangal is investigated.

%%%%%%%%%%%%%%%%%%%%%
%	Section
%%%%%%%%%%%%%%%%%%%%%

\section{General definitions and notations}
  	\label{sec:definitions}
  	  	
  	The term \textit{function}  denotes a mapping $ f: \mathbb{R} \mapsto \mathbb{R} $ (or in some cases $\mathbb{C} \mapsto \mathbb{C}$). 
  	The notation $f(x)$ is used to refer to the value of the mapping at the point \textit{x}.
    The term \textit{operator}  denotes the mapping from functional expressions to functional expressions.
  	Square brackets are used for the arguments of operators, while round brackets are used for the arguments of functions.	
    $Dom[f]$ denotes the domain of definition of the function $f(x)$.
    The term Cauchy sequence will be always interpreted  as a null sequence.
 
  	 %%%%%%%%%
  	 % Definition
  	 %%%%%%%%%
  	 \begin{definition}[Asymptotic $\smallO$ notation]
   	 	\label{def:bigoh}
   	 	The notation  $\bigoh{x^\alpha}$ is interpreted as the convention that 
   	 	$
   	 	\llim{x}{0}{ \frac{\bigoh{x^\alpha}}{x^\alpha} } =0 
   	 	$ 
   	 	for $\alpha >0 $.
   	 	The notation $\bigoh{1}$ is interpreted as a Cauchy-null sequence.
   	 	% so that    	 	$\llim{x}{0}{ \frac{\bigoh{x^\alpha}}{x^\alpha} } - \bigoh{1} =0 $.
	\end{definition}
   	 %%%%%%%%%
   	 % Definition
   	 %%%%%%%%%
   	 \begin{definition}
   	 	\label{def:holder}
   	 	%Let H\"older \holder{\beta} be the class  of  $\mathbb{C}^0$  functions of fractional order $\beta$, $\beta \in (0,\, 1]$, %That is, 
   	 	We say that $f$ is of (point-wise) H\"older  class \holder{\beta} if for a given $x$ there exist two positive constants 
   	 	$C, \delta \in \mathbb{R} $ that for an arbitrary $ y \in Dom[ f ]$ and given $|x-y| \leq \delta$ fulfill the inequality
   	 	$
   	 	| f (x) - f (y) |  \leq C |x-y|^\beta
   	 	$, where $| \cdot |$ denotes the norm of the argument.  	
   	 	
   	 	For (mixed) orders $n +\beta $ ($n \in \mathbb{N}_0$) the  H\"older class \holder{n+ \beta}  designates the functions  for which the inequality 
   	 	\[
   	 	| f (x)  - P_n (x-y) |  \leq C |x-y|^{n +\beta}  
   	 	\]
   	 	holds under the same hypothesis for $C, \delta $ and \textit{y}. $P_n (.)$ designates the polynomial  
   	 	$
   	 	P_n (z) = f (y)+ \sum\limits_{k=1}^{n}{ a_k z^k} 
   	 	$.
   	 \end{definition}
   	 %%%%%%%%%%%%%%%%%%%%%%%%%%%5
   	 \begin{remark}
   	 	The polynomial $P_n(x)$ can be identified with the Taylor polynomial of order $n$ of $f(x)$
   	 	(see for example  \cite{Prodanov2015} ).	
   	 \end{remark}

  	%%%%%%%%%%%%%%%%%%%%%%%%%%
  	% Definition
  	%%%%%%%%%%%%%%%%%%%%%%%%%%	 
  	\begin{definition}
  		\label{def:deltas}
  		Let the parametrized difference operators acting on a function $f(x)$ be defined in the following way
  		\begin{flalign*}
	  	\deltaplus{f}{x}   & :=  f(x + \epsilon) - f(x) \ecma\\
  		\deltamin{f}{x} & :=  f(x) - f(x - \epsilon)  %\ecma \\
  		%\deltaop {f}{x}{\epsilon}{2}  &:=  f(x + \epsilon) -2 f(x) + f(x - \epsilon) \ecma
  		\end{flalign*}
  		where $\epsilon>0$. The first one we refer to as \textit{forward difference} operator, 
  		the second one we refer to as \textit{backward difference} operator and the third one as \textit{2\textsuperscript{nd} order difference} operator.
  	\end{definition}
  	
  	%%%%%%%%%%%%%%%%
  	% Section
  	%%%%%%%%%%%%%%%%%
 	\section{Point-wise oscillation of functions}
  	\label{sec:osc}
  	
  	The concept of point-wise oscillation is used to characterize the set of continuity of a function.
  	To this end I build further on a technical result, which is presented as a Theorem 3.5.2 in Trench \cite{Trench2013}[p. 173]. 
  	Here the proof is slightly modified to account for separate treatment of right- and left- continuity.
  	 
  	%%%%%%%%%%%%%%%%%%%%%%
  	%  Definition
  	%%%%%%%%%%%%%%%%%%%%%%
  	\begin{definition}
  		\label{def:limosc}
		Define forward oscillation and its limit as the operators
	    \begin{flalign*}
	        	%\mathrm{osc_{\epsilon}^{+}} [f] (x) 
	        	\osc{f}{x}{\epsilon}{+} : =  & \sup_{[x , x + \epsilon]} {[ f]} -
	        		\inf_{[x , x + \epsilon]} {[ f]} \\
	     	\mathrm{osc^{+}} [f] (x) : = & \llim{\epsilon}{0}{ \left( \sup_{[x , x + \epsilon]}   -
	     		\inf_{[x , x + \epsilon]} \right) f} = \llim{\epsilon}{0}{\osc{f}{x}{\epsilon}{+}}
	    \end{flalign*}
 		and backward oscillation and its limit as the operators
		\begin{flalign*}
			%\mathrm{osc_{\epsilon}^{-}} [f] (x) 
			\osc{f}{x}{\epsilon}{-} : =  & \sup_{[x - \epsilon , x ]} {[ f]} -
				\inf_{[x -   \epsilon, x ]} {[ f]} \\
		\mathrm{osc^{-}} [f] (x) : = & \llim{\epsilon}{0}{ \left( \sup_{[x - \epsilon , x ]}   -
			\inf_{[x -   \epsilon, x ]}\right)  f} = \llim{\epsilon}{0}{\osc{f}{x}{\epsilon}{-}}
		\end{flalign*}
		according to previously introduced notation \cite{Prodanov2015}.
  	\end{definition}
  	
  	%%%%%%%%%%%%%%%%%%%%%%%%%%%%
	
	 %%%%%%%%%%
	 % Theorem
	 %%%%%%%%%%
	 \begin{lemma}[Oscillation lemma]
	 \label{th:osc}
	 
	 Let $I_{+} :=\left[x, x+ \epsilon \right] \subseteq Dom[f]$  then $f$  is \underline{right-continuous} for $x \in I_{+} $ iff  $\mathrm{osc^{+}} [f] (x)=0 $.	 
	 Alternatively, 
	 Let $I_{-} :=\left[ x - \epsilon, x \right] \subseteq Dom[f]$  then $f$  is \underline{left-continuous} for $x \in I_{-} $ iff  $\mathrm{osc^{-}} [f] (x)=0 $.
	 \end{lemma}
	 %%%%%%%%%%%%%%%%%%%%%%%%%%%%%%
	 %%%%%%%%%%%%%%%%%%%%%%%%%%%%%%%%%%%%%
	 \begin{proof}
	 	\begin{description}
	 	\item[Forward case]  
		 Suppose that  $\mathrm{osc^{+}} [f] (x)=0 $. 
		 We select a variable $x_0  \in \left[x, x+ \epsilon \right] $. 
		 Then by continuity we have 
		 \[
		\inf_\epsilon f (x) \leq f(x_0)	 \leq \sup_\epsilon f (x) 
		 \]
		 so that
		  \[
		  0 \leq f(x_0)	-  \inf_\epsilon f (x) \leq \sup_\epsilon f (x) -\inf_\epsilon f (x) = \mathrm{osc_{\epsilon}^{+}} [f] (x) < \mu
		  \]
		  by hypothesis for some positive $\mu  $  
		and  $h = x^{\prime} -x_0  \leq \epsilon $ assuming  $\inf_\epsilon f (x) \equiv  f (x^{\prime} )$.
		 But then since $x_0$ is free we can set $x_0 \mapsto x$, therefore
		 \[
		  0 \leq f(x )	-  \inf_\epsilon f (x) < \mu
		 \]
		 for $h = x^{\prime} -x   \leq \epsilon $.

		 In a similar way,
		 \[
		 \sup_\epsilon f (x) - \inf_\epsilon f (x) =   \sup_\epsilon f (x) - f(x ) + f(x )	-  \inf_\epsilon f (x) < \mu
		 \]
		 implying
		 \[
		 0 \leq \sup_\epsilon f (x) - f(x ) < \mu - ( f(x )	-  \inf_\epsilon f (x)) < \mu
		 \]		   
		 Therefore, we can select Cauchy sequences $\epsilon$ and $\mu < \epsilon N $ by the Archimedean property for some $N$. 
		 Therefore, \textit{f} is right continuous at \textit{x}\footnote {
		 This argument asserts the equalities 
		 $
		 \llim{\epsilon}{0}{  \inf_{[x, x+ \epsilon]} f(x) } = \llim{x_0}{x^+}{ f(x)} = \llim{\epsilon}{0}{  \sup_{[x, x+ \epsilon]} f(x) } ,
		 $
		 which is the equivalent to the \textit{Squeeze lemma}}. 
	 %%%%%%%%%%%%%%%%%%%%%%%%%%%%%%%%%%%%%
	 
	 \item[Reverse case]
		 If \textit{f} is (right-) continuous about \textit{x} then there exist related (Cauchy sequences) $\mu$ and $\delta$ such that		 
		 \begin{flalign*}
		 \left|  f(x^{\prime}) - f(x)\right| & < \mu/2,  \ \  \left| x^{\prime}  -  x \right|  < \delta/2 \\
		 \left|  f(x) - f(x^{\prime\prime}) \right|  & < \mu/2,  \ \ \left|   x -  x^{\prime\prime}    \right| < \delta/2
		 \end{flalign*}
		Then we add the inequalities and by the triangle inequality we have
	 \begin{flalign*}
		\left|  f(x^{\prime}) - f(x^{\prime\prime})\right| & \leq \left|  f(x^{\prime}) - f(x)\right|+ \left|  f(x) - f(x^{\prime\prime}) \right|  < \mu \\ 
		\left| x^{\prime}  -  x^{\prime\prime}  \right| & \leq \left| x^{\prime}  -  x \right| + \left|   x -  x^{\prime\prime}    \right| < \delta \epnt
	\end{flalign*}
		However, since $ x^{\prime}$ and $x^{\prime\prime}$ are arbitrary we can set the former to correspond to the minimum and the latter to the maximum of \textit{f} in the interval.
		therefore, by the least-upper-bond property we can map 
		  $ f(x^{\prime}) \mapsto \inf_\epsilon f (x) $, $f(x^{\prime\prime}) \mapsto \sup_\epsilon f (x) $.		
		Therefore, 
		$ \mathrm{osc_{\delta}^{+}} [f] (x) < \mu $  for   $ \left|  x^{\prime} - x^{\prime\prime}  \right| < \delta $ (for the same $\mu$ and $\delta$ ).
		Therefore, the limit is $\mathrm{osc^{+}} [f] (x)=0 $.
	\end{description}
    The left case follows by applying the right case, just proved, to the mirrored image of the function $ f(-x)$.
	 \end{proof}
	 %%%%%%%%%%%%%%%%%%%%%%%%%%%%%%%%%%%%%%%%%
	Therefore, we have the obvious corollary that if $f$ is both right- and left- continuous iff $\mathrm{osc^{+}} [f] (x)=\mathrm{osc^{-}} [f] (x) = \mathrm{osc} [f] (x)=0$, which was the actual statement in \cite{Trench2013}.
   	%%%%%%%%%%%%%%%%%%%%%%%%%%%
   	% Section
   	%%%%%%%%%%%%%%%%%%%%%%%%%%%
   	\section{Fractional (fractal) variations and fractional velocities}
   	\label{sec:frdiff}
 
  %%%%%%%%%%%%%%%%%%%%%5
  %	Definition
  %%%%%%%%%%%%%%%%%%%%%
  \begin{definition}
  	\label{def:fracvar}
  	Define \textit{Fractal Variation} operators of order $0 \leq \beta \leq 1$ as
  	\begin{align}
  	\label{eq:fracvar1}
  	\fracvarplus {f}{x}{\beta} := \frac{ \deltaplus{f}{x}  }{\epsilon ^\beta}  =\frac{ f(x+ \epsilon) - f(x) }{\epsilon ^\beta} 
  	\\
  	\fracvarmin {f}{x}{\beta} :=  \frac{ \deltamin{f}{x} }{\epsilon ^\beta}  =\frac{ f(x)- f( x- \epsilon)  }{\epsilon  ^\beta} 
  	\end{align}
  	for a positive  $\epsilon$.
  \end{definition}
%%%%%%%%%%%%%%%%%%%%%%%%%%%%%%%%%%%%%%%%%%%%%%%%%%%  
   
      %%%%%%%%%
      % Definition
      %%%%%%%%%  
       \begin{definition}[Fractional order velocity]
       	\label{def:frdiff}
       	Define  the \textsl{fractional velocity} of fractional order $\beta$ as the limit  
       	\begin{align}
       	\label{eq:fracdiffa}
       	\fdiffpm {f}{x}{\beta} &:= \llim{\epsilon}{0}{\frac{\Delta^{\pm}_{\epsilon} [f ] (x) }{\epsilon ^\beta}} 
       	=\llim{\epsilon}{0}{\fracvarpm {f}{x}{\beta}} \epnt       	   
       	\end{align}
       	%where  $\epsilon >0$ and  $0 < \beta \leq 1 $ are real parameters and $f(x)$ is function.	
       	A function for which at least one of \fdiffpm {f}{x}{\beta} exists finitely will be called $\beta$-differentiable at the point \textit{x}.
       \end{definition}
       %%%%%%%%%%%%%%%%%%%%%%%%%%%%%%%%%%%%%  
       In the above definition we do not require upfront equality of left and right $\beta$-velocities. 
       This amounts to not demanding upfront continuity of the   $\beta$-velocities (see. Prop. \ref{prop:cont1} ). 
     	
     	%%%%%%%%%%%%%
   		% Definition
   		%%%%%%%%%%%%%
    	\begin{definition}
    	The set of points where the fractional velocity exists finitely and $\fdiffpm{f}{x}{\beta}  \neq 0 $ will be denoted as the \textbf{set of change}
  	   	$ 
  	   	\soc{\pm}{\beta} (f): = \left\lbrace x: \fdiffpm{f}{x}{\beta}  \neq 0\right\rbrace 
  	   	$.
    	\end{definition}  
      	
       We are ready to establish the existence conditions of the fractional velocity.
       %To this end we need to state two pivotal conditions:
       To this end we will find it helpful to formulate the following two pivotal conditions:
       %%%%%%%%%%%%%%%%%%%%%%%%%%%%%%%%% 
        \begin{condition}[H\"older growth condition]
        For  given  $x$ and $0< \beta \leq 1$ 
           	\begin{equation}\label{C1} 
           	\mathrm{osc}_{\epsilon }^{\pm}f (x)  \leq C \epsilon^\beta \tag{C1}
           	\end{equation}
           	for some  $C \geq 0$ and $\epsilon > 0$.    
        \end{condition}
       %%%%%%%%%%%%%%%%%%%%%%%%%%%%%
       \begin{condition}[H\"older oscillation condition]
       	%[Vanishing fractional variation oscillation condition]\label{C2}    
        For  given  $x$ and $0< \beta \leq 1$  	
       	\begin{equation}\label{C2}
       	\mathrm{osc}^{\pm} \fracvarpm {f}{x}{\beta} =0 \epnt \tag{C2}
       	\end{equation}
       \end{condition}
       %%%%%%%%%%%%%%%%%%%%
 
       %%%%%%%%%%%%%%%%%
       %  Theorem
       %%%%%%%%%%%%%%%%%%
       \begin{theorem}[Conditions for existence of $\beta$-velocity]\label{th:aexit}
       For each $\beta > 0$ if \fdiffplus {f}{x}{\beta} exists (finitely), then $f$ is right-Holder continuous of order $\beta$ at $x$, and the analogous result holds for \fdiffmin {f}{x}{\beta}and left-Holder continuity.
  
       Conversely, if conditions \ref{C1}  and \ref{C2}  hold then $\fdiffpm {f}{x}{\beta}$ exists.
       Moreover, the H\"older oscillation condition is a necessary and sufficient condition for the existence of $\beta$-velocity.
       The H\"older growth condition is a necessary condition for the existence of $\beta$-velocity.      
       \end{theorem}
       %%%%%%%%%%%%%%%%%%%%%%%%%%
       %%%%%%%%%%%%%%%%%%%%%%%%%%%%%%%
       \begin{proof}  
       	We will first prove  the case for right continuity.
      Condition C1 trivially implies the usual H\"older growth condition, given according to our notation as  
      $\fracvarpm {f}{x}{\beta}  \leq C \epsilon^\beta$.   
       	\begin{description}
       		\item[Forward case]
       	
       	Let $L>0$ be the limit value. Then by hypothesis we have
       	\[
       \left| 	\frac{\deltaplus{f}{x}}{\epsilon^\beta} - L \right|  < \mu
       	\]
       	for a positive  Cauchy sequence   $\mu$ implying existence of a Cauchy sequence $ \epsilon < \delta$. 
       	Straightforward rearrangement gives
       	\[
       	\left| 	f(x+ \epsilon) - f(x) - L \epsilon^\beta \right|  < \mu \epsilon^\beta \epnt
       	\]
       	Then by the reverse triangle inequality
   		\[
   		\left| 	f(x+ \epsilon) - f(x)   \right| - L \epsilon^\beta \leq
   		\left| 	f(x+ \epsilon) - f(x) - L \epsilon^\beta \right|  < \mu \, \epsilon^\beta \epnt
   		\]
   		Therefore, 
   		\[
   			\left| 	f(x+ \epsilon) - f(x)   \right| < \left( \mu + L  \right)  \epsilon^\beta \epnt
   		\]
   		Therefore, we can assign a Cauchy sequence to $\delta$ and demand that RHS approaches arbitrary close to 0 implying also $	\mathrm{osc^{+}} [f] (x) =0$.
   		In addition (i.e. by the least-upper-bound property) there is a number $C$, such that
   		\[
   		\left| 	f(x+ \epsilon) - f(x)   \right| \leq C \epsilon^\beta \ecma
   		\]
   		which is precisely the H\"older growth condition. The left continuity can be proven in the same way.
   		%%%%%%%%%%%%%%%%%%%%%%%%%%%%%%%%%%%
   		\item[Reverse case]
  		In order to prove the converse statement we can observe that the first hypothesis implies that 	
	\[
	\left| 	 \deltaplus{f}{x}    \right| \leq   \mathrm{osc}_{ \epsilon }^{+}f (x) 
	   \leq C \epsilon^\beta \ecma
	\]
	which in turn implies boundedness of the limit if it exists (i.e. a necessary condition).
	Next, we observe that 
	\[
	\liminf\limits_{\epsilon \rightarrow 0 } \frac{\deltaplus{f}{x}     }{\epsilon^\beta} 
	  = \limsup\limits_{\epsilon \rightarrow 0 } \frac{\deltaplus{f}{x}     }{\epsilon^\beta} \Leftrightarrow
	  \lim\limits_{\epsilon \rightarrow 0 } \frac{\deltaplus{f}{x}     }{\epsilon^\beta}
	\]
	for the RHS limit to exist.
	But since \fracvarplus{f}{x}{\beta}   is continuous in $\epsilon>0$ the equality of the limits by 
	Lemma \ref{th:osc} implies that 
		\[
			\lim\limits_{\epsilon \rightarrow 0 }\left(  \sup\limits_{\epsilon } -\inf\limits_{\epsilon }  \right) \frac{\deltaplus{f}{x} }{\epsilon^\beta} =
		\lim\limits_{\epsilon \rightarrow 0 } \mathrm{osc}^{+}_{\epsilon } \frac{\deltaplus{f}{x}     }{\epsilon^\beta} =0 \epnt
		\]
	Then under the adopted notation
			\[
			\lim\limits_{\epsilon \rightarrow 0 } \mathrm{osc}^{+}_{\epsilon } \frac{\deltaplus{f}{x}     }{\epsilon^\beta} =  \mathrm{osc}^{+} \fracvarplus {f}{x}{\beta} =0 \ecma
			\]
	which is  the condition C2, i.e. vanishing of fractional variation oscillation.
	
	In order to establish the sufficiency assertion we can observe that if the limit exists then the reasoning from the forward case  applies. 
	Indeed, in this case let's set $\fdiffplus {f}{x}{\beta} =a$.
	Then we have 
	\[
	\left| \frac{\Delta_\epsilon^{+} f (x) - a \epsilon^\beta }{\epsilon^\beta} \right| < \mu \epnt
	\]
	Then it follows that 
	$
	\left|  \Delta_\epsilon^{+} f (x) \right| < (\mu +a ) \epsilon^\beta  
	$	
	which is the usual H\"older growth condition. 
\end{description}
	
    The left case follows by applying the right case, just proved, to the function $g(x) = f(-x)$.
    \end{proof}
%%%%%%%%%%%%%%%%%%%%%%%%%%%%%%%%%%%%%%%%%%%%%% 
   
    This result is a generalization of Cherbit \cite{Cherbit1991}, where the first part of the Theorem \ref{th:aexit} above     is stated without proof. 
    The result further indicates towards splitting the (point-wise) Holder class of functions into two disjoint sub-classes -- a regular 
    \holder{r, \beta} class, where the H\"older oscillation condition is fulfilled and the function is $\beta$-differentiable, and an oscillatory 
    \holder{\star, \beta} class, where the  H\"older  growth condition is fulfilled but the oscillation condition fails and the function is not $\beta$-differentiable. 
    To demonstrate the latter, let us consider the following example:
    %%%%%%%%%%%%%%%%%%%%%%%%%%%%%%%%%
    \begin{example}
	 The function \[
	 f(x):= \left\{
	 \begin{array}{ll}
	 0  ,& x=0  \\
	  \sqrt{x} \sin \left( \frac{1}{x}\right)   ,&  x > 0  
	 \end{array}
	 \right.
	 \]
	 has a strong oscillation at $x \rightarrow 0$. 
	 We will compute the fractional velocity at $x=0$ at the critical order $\beta=1/2$:
	 \[
	 \fracvarplus{f}{0}{1/2} = \sin \left( \frac{1}{\epsilon}\right)   
	 \]
	 It can be established that $\mathrm{osc}^{+} \fracvarplus{f}{0}{1/2} = 2 $. Therefore, the function fails the vanishing oscillation condition and its $1/2$-velocity is undefined at $x=0$.
	 On the other hand, for $\beta < 1/2$  we have
	  $
	  \fracvarplus{f}{0}{\beta} = \epsilon^{\frac{1}{2} - \beta } \sin \left( \frac{1}{\epsilon}\right)    
	  $.
	Therefore, 	$ \fdiffplus{f}{0}{\beta} = 0 	$.
    \end{example}
    %%%%%%%%%%%%%%%%%%%%%%%%%%%%%%%%%%%%%%%%%%%%%%%%%%%%%%%

In addition, Th. \ref{th:aexit} demonstrates that the most useful interpretation of the H\"older exponents is as a point-wise property of a function. 
Therefore, we will interpret the statement $ f(x) \in \holder{\beta}$ as "the function belongs to the H\"older class at the point \textit{x}".
 
%%%%%%%%%%%%%%%%%%%
%   Corr
%%%%%%%%%%%%%%%%%%%%
\begin{corollary}[Fractional approximation property]
	\label{prop:bigoh}
	If  \textit{f} is $\beta$-differentiable about $x$ in the interval $[x, x+ \epsilon]$ (resp. $[ x - \epsilon, x]$ ) 
	for all $\epsilon$ (uniformly)
	then
	\[
	\left| \frac{\Delta_\epsilon^{\pm} f (x) - \fdiffpm {f}{x}{\beta} \epsilon^\beta }{\epsilon^\beta} \right| = \bigoh{1} \leq \gamma
	\]
	for some $\gamma>0$.
\end{corollary}
%%%%%%%%%%%%%%%%%%%%%%%%%%%%%%%%%%%%%
\begin{proof}
	We prove the forward case first.
	Consider the interval  $[x, x+ \epsilon]$.
	Then by Th. \ref{th:aexit} 
	\[
	L= \left| \frac{\Delta_\epsilon^{+} f (x) - \fdiffplus {f}{x}{\beta} \epsilon^\beta }{\epsilon^\beta} \right| 
	\]
	is a Cauchy sequence therefore, under the $\smallO$-notation  $L =  \bigoh{1}$.
	By the least-upper-bond property there is $\gamma = \sup{L}$, therefore, $L \leq \gamma$.
	The proof of the backward case follows by considering the  interval  $[ x - \epsilon, x]$ and applying the same argument to $\Delta_\epsilon^{-} f (x)$ and $ \fdiffmin {f}{x}{\beta} $, respectively.
\end{proof}
%%%%%%%%%%%%%%%%%%%%%%%%%%%%%%

Further, using so-established existence result it is straightforward to demonstrate the fractional Taylor expansion property:

%%%%%%%%%%%%%%%%%%%%%%%%%%
% Theorem
%%%%%%%%%%%%%%%%%%%%%%%%%%
\begin{corollary}[Fractional Taylor-Lagrange property]
	\label{th:holcomp1}
	The existence of $\fdiffpm{f}{x}{\beta} \neq 0$  for  $\beta \leq 1$ implies that
	\begin{equation}
	f(x \pm \epsilon) = f(x) \pm \fdiffpm {f}{x}{\beta}   \epsilon^\beta + \bigoh{\epsilon^{\beta} }  \epnt
	\end{equation}
	While if  
	\[
	f(x \pm \epsilon)= f(x) \pm K \epsilon^\beta +\gamma \; \epsilon^\beta  
	\] 
	uniformly in   the interval $ x \in [x, x+ \epsilon]$ for some $\gamma$, such that \llim{\epsilon}{0}{ \gamma}=0 then \fdiffpm {f}{x}{\beta} = K.
\end{corollary}
%%%%%%%%%%%%%%%%%%%%%%%%%%%%%%%%%%%%%%%%%%%%%
\begin{proof}
	\begin{description}
		\item[Forward statement]  
			Suppose that 
			\[
			f(x +\epsilon)= f(x) + K \epsilon^\beta +\gamma    \ecma
			\] 
			where $ K= \fdiffpm{f}{x}{\beta}$ and  $\gamma = \bigoh{\epsilon^{\beta} }  $.
			Then 
			$\fracvarplus{f}{x}{\beta} = K + \bigoh{1} $. Taking the limit provides the result.
			The backward case is proven in a similar manner.

		\item[Converse statement]
			Suppose that 
			\[
			f(x +\epsilon)= f(x) + K \epsilon^\beta +\gamma \; \epsilon^\beta  \ecma
			\] 
			uniformly in  the interval $ x \in [x, x+ \epsilon]$.
			Then this fulfills both H\"older growth and vanishing oscillation conditions. 
			Therefore, 	$ K = \fdiffpm{f}{x}{\beta}$ observing that \llim{\epsilon}{0}{ \gamma}=0.
	\end{description}
\end{proof}
%%%%%%%%%%%%%%%%%%%%%%%%%%%%%%%%%%%%%%%%%%%%%%%%%
This result implies that regular H\"older functions can be approximated locally as fractional powers of appropriate order.

%  %%%%%%%%%%%%%%%%%%%%%
  %  Lemma
  %%%%%%%%%%%%%%%%%%%%%%
  \begin{lemma}[Bounds of forward variation]
  	\label{th:bondvar1}
  	Let   $\fdiffplus {f}{x}{\beta} \neq0$ and
  	$C_x$ and $C^{\prime}_{x}$ be constants such that 
  	i)  $C$ is the smallest number for which $|\Delta_{\epsilon}^{+}[f] (x)| \leq C \epsilon^\beta$ still holds, 	that is
  	$C_x =  \underset{C}{\inf} \{  |\Delta_{\epsilon}^{+}[f] (x)| \leq C \epsilon^\beta\}$ 	and 
  	ii) $C^{\prime}_x$ is the largest number $C$ for which  $ |\Delta_{\epsilon}^{+}[f] (x)| \geq C \epsilon^\beta$
  	still holds, that is 
  	$C^{\prime}_{x} =  \underset{C}{\sup} \{  |\Delta_{\epsilon}^{+}[f] (x) | \geq C \epsilon^\beta\}$.
  	Then
  	\[
  	|\fdiffplus {f}{x}{\beta} | = C_x =C^{\prime}_{x} \epnt
  	\]
  \end{lemma}
  %%%%%%%%%%%%%%%%%%%%%%%%%%%%%%%%%%%%%%%%%%%%%%%%
  \begin{proof}
  	We fix $x$ so that $C_x$ remains constant. Let $ \fdiffplus {f}{x}{\beta} =q$ for a certain real number \textbf{q}.
  	From the definition of a H\"{o}lder function it follows that $ |\Delta_{\epsilon}^{+}[f] (x) | \leq C_x \epsilon^\beta$ then we subtract $ q \epsilon^\beta $ from both sides of the inequality to obtain
  	\[
  	\left( C^{\prime}_{x} -q \right) \epsilon^\beta \leq	|\Delta_{\epsilon}^{+}[f] (x) | -q \epsilon^\beta \leq \left( C_x -q \right) \epsilon^\beta \epnt
  	\]
  	Further, division by the positive quantity $  \epsilon^\beta$ results in 
  	\[
  	C^{\prime}_{x} -q  \leq \frac{|\Delta_{\epsilon}^{+}[f] (x) | -q \epsilon^\beta}{\epsilon^\beta} \leq   C_x -q  \epnt
  	\]
    We now consider two cases.
  	If $\Delta_{\epsilon}^{+}[f] (x) \geq 0 $  then taking the limit of both sides gives
  	\[
  	\llim{\epsilon}{0}{ \frac{\Delta_{\epsilon}^{+}[f] (x)  -q \epsilon^\beta}{\epsilon^\beta} } =0
  	\] 
  	and
  	$ \llim{\epsilon}{0}{ C_x -q} \geq 0$ but since $C_x$ is an infimum then $ \llim{\epsilon}{0}{ C_x  } = q$.
  	Since $C$ is constant with regard to $\epsilon$ it follows that $C_x =q$.
  	On the other hand, if $\Delta_{\epsilon}^{+}[f] (x) \leq 0 $ then
  	\[
  	\llim{\epsilon}{0}{ \frac{ - \Delta_{\epsilon}^{+}[f] (x)  -q \epsilon^\beta}{\epsilon^\beta} } = - 2 q
  	\] and
  	$ \llim{\epsilon}{0}{ C_x -q} \geq -2 q$, therefore $ \llim{\epsilon}{0}{ C_x } \geq -q$ and the same reasoning as in the previous case applies to yield $C_x= - q$.
  	Therefore, finally $C_x =|q|$.
  	The case for $C^{\prime}_{x}$ can be derived using identical reasoning. 
  \end{proof}
  %%%%%%%%%%%%%%%%%%%%%%%%%%%
  %%%%%%%%%%%%%%%%%%%%%
  %  Lemma
  %%%%%%%%%%%%%%%%%%%%%%
  \begin{lemma}[Bounds of backward variation]
  	\label{th:bondvar2}
  	Let $\fdiffmin {f}{x}{\beta} \neq0$ and
  	$C_x$ and $C^{\prime}_{x}$ be constants such that 
  	i)  $C$ is the smallest number $C$ for which $|\Delta_{\epsilon}^{-}[f] (x)| \leq C \epsilon^\beta$ still holds, 	that is
  	$C_x =  \underset{C}{\inf} \{  |\Delta_{\epsilon}^{-}[f] (x)| \leq C \epsilon^\beta\}$ 	and 
  	ii) $C^{\prime}_x$ is the largest number $C$ for which  $ |\Delta_{\epsilon}^{-}[f] (x)| \geq C \epsilon^\beta$
  	still holds, that is 
  	$C^{\prime}_{x} =  \underset{C}{\sup} \{  |\Delta_{\epsilon}^{-}[f] (x) | \geq C \epsilon^\beta\}$.
  	Then
  	$
  	|\fdiffmin {f}{x}{\beta} | = C_x =C^{\prime}_{x} 
  	$.
  \end{lemma}
  %%%%%%%%%%%%%%%%%%%%%%%%%%%%
%Collectively these results imply the following proposition 

%%%%%%%%%%
%  Section
%%%%%%%%%%%%
\section{Continuity of fractional velocity}
\label{sec:cont}
 
Gleyzal \cite{Gleyzal1941} established that a function is Baire-one if and only if it is the limit of an interval function.
Therefore, \fdiffpm {f}{x}{\beta} are Baire class 1 from which it follows that \fdiffpm {f}{x}{\beta} must be continuous on a dense set. 
%By the Baire Category theorem  for \fdiffpm {f}{x}{\beta} the set of points of discontinuity of \fdiffpm {f}{x}{\beta} is either meager or else has nonempty interior.
%Further, we can give a more precise result about the set of discontinuity of \fdiffpm {f}{x}{\beta}.
Moreover, since the continuity set of a function is a $G_{\delta}$ set (i.e. an intersection of at most countably many open sets), 
it follows from the Osgood-Baire Category theorem that the set of points of discontinuity of
\fdiffpm {f}{x}{\beta} is meager (i.e. a union of at most countably many nowhere dense sets).
Further, we can give a more precise result about the set of discontinuity of \fdiffpm {f}{x}{\beta}.

         %%%%%%%%%
         % Proposition
         %%%%%%%%%  
         \begin{proposition}
         	\label{prop:cont1}
         	Suppose that  $\fdiffpm {f}{x}{\beta}$ exist finitely and are continuous. 
         	Then $\fdiffplus {f}{x}{\beta} = \fdiffmin {f}{x}{\beta}$.
         	The converse is not always true.
         \end{proposition}
         %%%%%%%%%%%%%%%%
         \begin{proof}   	
         	
         	Suppose that  $\fdiffpm {f}{x}{\beta}$ exist finitely and are continuous.
         	Then the double-sided limit
         	\[
         	D^\beta f (x)  = \lim_{\varepsilon\rightarrow 0}\frac{f(x+\varepsilon)-f(x-\varepsilon)}{\left(  2 \, \varepsilon\right) ^\beta} = 
         	\lim_{\varepsilon\rightarrow 0} \frac{\fracvarplus{f}{x}{\beta} +\fracvarmin{f}{x}{\beta}}{2^\beta}
         	\]
         	exists.  
         	Set  $y=x -\varepsilon$ and the same calculation gives
         	\begin{flalign*}
         	D^\beta f (x) & = \lim_{\varepsilon\rightarrow 0}\frac{f(x+\varepsilon)-f(x-\varepsilon)}{\left(  2 \, \varepsilon\right) ^\beta} 
         	= \lim_{\varepsilon\rightarrow 0}\frac{f( y) - f(y - 2 \varepsilon)}{\left(  2 \, \varepsilon\right) ^\beta} \\
         	& = \fdiffmin{f}{y}{\beta}= \lim_{\varepsilon\rightarrow 0}  \fdiffmin{f}{x - \varepsilon}{\beta}=  \fdiffmin{f}{x }{\beta}
         	\end{flalign*}
         	In the last step of the above argument we use the hypothesis that  \fdiffmin{f}{x}{\beta} is continuous about $x_{-}$.
         	Further, setting  $z=x +\varepsilon$ gives 
         	\begin{flalign*}
         	D^\beta f (x) & = \lim_{\varepsilon\rightarrow 0}\frac{f(x+\varepsilon)-f(x-\varepsilon)}{\left(  2 \, \varepsilon\right) ^\beta} 
         	= \lim_{\varepsilon\rightarrow 0}\frac{f( z+2  \varepsilon) - f(z)}{\left(  2 \, \varepsilon\right) ^\beta} \\
         	& =  \fdiffplus{f}{z}{\beta}= \lim_{\varepsilon\rightarrow 0}  \fdiffmin{f}{x +\varepsilon}{\beta} = \fdiffplus{f}{x}{\beta}
         	\end{flalign*}
         	using the continuity of $\fdiffplus {f}{x}{\beta}$.
         	Therefore, $ \fdiffplus{f}{x}{\beta}= \fdiffmin{f}{x}{\beta}$.
         	
         	Finally, consider the function
         	\[
         	g(x):=\begin{cases}
         	x^\beta, & x \geq 0 \\
         	-|x|^\beta, & x <0
         	\end{cases}
         	\] for $0 < \beta < 1$.
         	Direct computations show that 
         	\[
         	\fdiffplus{g}{x}{\beta}:=\begin{cases}
         	1, & x = 0 \\
         	0, & x \neq 0
         	\end{cases}
         	\]
         	and
         	\[
         	\fdiffmin{g}{x}{\beta}:=\begin{cases}
         	1, & x = 0 \\
         	0, & x \neq 0
         	\end{cases}
         	\]  
         	which are discontinuous about $x=0$ but $ \fdiffplus{g}{0}{\beta}=\fdiffmin{g}{0}{\beta}$.
         \end{proof}
         %%%%%%%%%%%%%%%%%%%%%
         
%%%%%%%%%%%%%%%%%%%%%%%%%%
% Prop
%%%%%%%%%%%%%%%%%%%%%%%%%% 
%%%%%%%%%%%%%%%%%%%%%%%%%%%%%
\begin{proposition}\label{th:diffvar2}
 	Let $f(x) \in \holder{r, \beta}$ then for all $0<\alpha < \beta \leq 1$ $ \fdiffplus{f}{x}{\alpha} =  \fdiffmin{f}{x}{\alpha}   = 0	$.
\end{proposition}   
   
  %%%%%%%%%%%%%%%%%%%%%%%%%%
  % Theorem
  %%%%%%%%%%%%%%%%%%%%%%%%%% 
  \begin{theorem}[Discontinuous velocity]\label{th:discontd}
  	Let $f(x) \in \holder{r, \beta}$ for $\beta <1$. 
  	If $\fdiffpm {f}{x}{\beta} \neq 0$ then $\fdiffpm {f}{x}{\beta}$ is discontinuous about $x$ and \soc{+}{\beta} is totally disconnected.
  \end{theorem}
  % % % % % % % % % % % % % % %
  \begin{proof}
  	Let's assume that \fdiffplus {f}{x}{\beta} is continuous (i.e. $\llim{\epsilon}{0}{ \fdiffplus {f}{x +\epsilon }{\beta} = \fdiffplus {f}{x}{\beta}}$). 
  	Let's fix a point \textit{x} and suppose that $\fdiffplus {f}{x}{\beta} = K \neq 0$.
  	Then in the interval $[x, x+ \epsilon]$
  	\[
  	\llim{\epsilon}{0}{ \dfrac{f(x+ \epsilon) - f(x+ \epsilon/2) + f(x+ \epsilon/2) - f(x)}{\epsilon^\beta}} = K \epnt
  	\]
  	However, by assumption of continuity $x + \epsilon/2 \in \soc{+}{\beta}$ and
  	\[
  	\llim{\epsilon}{0}{ \dfrac{f(x+ \epsilon) - f(x+ \epsilon/2)  }{\epsilon^\beta}} +\llim{\epsilon}{0}{ \dfrac{f(x+ \epsilon/2) - f(x)}{\epsilon^\beta}}= K \epnt
  	\]
  	Then also
  	\[
  	\llim{\epsilon}{0}{ \dfrac{f(x+ \epsilon) - f(x+ \epsilon/2)  }{2^\beta \left( \epsilon/2 \right)^\beta }} +\llim{\epsilon}{0}{ \dfrac{f(x+ \epsilon/2) - f(x)}{2^\beta \left( \epsilon/2 \right)^\beta }}= K \epnt
  	\]
  	Therefore,
  	$
  	\frac{2 K }{2^\beta} = K \epnt
  	$
  	Therefore, $\beta=1$, which is a contradiction. Therefore, \fdiffplus {f}{x}{\beta}  is not continuous and $x + \epsilon/2 \notin \soc{+}{\beta}$ for any $\epsilon >0$.
  	Since both $x$ and $\epsilon$ are arbitrary the same reasoning applies also for different point
  	$x^\prime$ and number $\epsilon$, therefore \soc{+}{\beta} is totally disconnected.
  	The same reasoning can be applied to \fdiffmin {f}{x}{\beta} for the backward case.
  \end{proof}
  %%%%%%%%%%%%%%%%%%%%%%%%%
   
  \begin{remark}
  	It can be noted that for $\beta=1$ the above argument is not sufficient to establish the discontinuity of $f^\prime(x)$.
  	Indeed the set $ D \subseteq \fclass{R}{}$ is the discontinuity set for some derivative if and only if \textit{D} is an $F_\sigma$ of the first category (i.e. a $F_\sigma$-meager) subset of \fclass{R}{} \cite{Bruckner1966}, that is a union of countable collections of closed subsets of \fclass{R}{}.\footnote{Answer by Mr. Dave L. Renfro at \url{https://math.stackexchange.com/questions/112067/how-discontinuous-can-a-derivative-be}.}
  \end{remark}
  %%%%%%%%%%%%%%%%%%%%%%%%%%
 The  argument given above can in fact be used to establish in general the continuity of the fractional velocity.
 
    %%%%%%%%%%%%%%%%%%%%%%
    %   Theorem
    %%%%%%%%%%%%%%%%%%%%% 

\begin{theorem}[Continuity of fractional velocity]
	\label{th:fdiffcont}
	For all $\beta<1$ if \fdiffpm {f}{x}{\beta} is continuous at \textit{x} then \fdiffpm {f}{x}{\beta}= 0.
\end{theorem}
 %%%%%%%%%%%%%%%%%%%%%%
 \begin{proof}
 	
 Let's leave $\beta$ unspecified and demand continuity of $\fdiffplus {f}{x}{\beta}$ in the interval $[x, x+ \epsilon]$ .
 Let's fix a point \textit{x} and suppose that $\fdiffplus {f}{x}{\beta} = K $.
 By Th. \ref{th:aexit} $f(x)$ is continuous in $[x, x+ \epsilon]$ and we can write
 \[
 \llim{\epsilon}{0}{ \dfrac{f(x+ \epsilon) - f(x+ \epsilon \lambda) + f(x+ \epsilon \lambda) - f(x)}{\epsilon^\beta}} = K  
 \]
 for a number $ 0 \leq \lambda \leq 1$.
 However, by assumption of continuity
 \[
 \llim{\epsilon}{0}{ \dfrac{f(x+ \epsilon) - f(x+ \epsilon \lambda)  }{\epsilon^\beta}} +\llim{\epsilon}{0}{ \dfrac{f(x+ \epsilon \lambda) - f(x)}{\epsilon^\beta}}= K \epnt
 \]
 Then if we restrict $\lambda$ to (0, 1) then
 \[
 \llim{\epsilon}{0}{ \dfrac{f(x+ \epsilon) - f(x+ \epsilon \lambda)  }{  \left( \epsilon \left( 1 - \lambda \right)  \right)^\beta / \left( 1 - \lambda\right)^\beta }} +\llim{\epsilon}{0}{ \dfrac{f(x+ \epsilon \lambda) - f(x)}{ \left( \epsilon \lambda \right)^\beta / \lambda^\beta}}= K \epnt
 \]
 Therefore 
 \[
 \left( 1 - \lambda \right)^\beta K +   \lambda^\beta K  = K  \iff   K \left( \left( 1 - \lambda \right)^\beta +  \lambda ^\beta    - 1 \right) =0 \epnt
 \]
 For the last equality to be true, we have either $ [\beta=1,   \forall \lambda]$ or for $\beta<1$ $ [\lambda=0, \lambda=1 ]$ if $K \neq 0$. 
 To demonstrate the last assertion we observe that for
 $\lambda=1/2$ we get
 $
 2^{1-\beta}=1
 $ 
 which is not true for $\beta \neq 1$.
 The case $\beta=1$ is excluded by hypothesis so then since $\lambda$ is arbitrary (but $\neq$ 0, 1 by construction) then $K=0$ must hold.
 \end{proof} 
\begin{remark}
Cresson proves some negative results corresponding to this theorem \cite[Th. 4.1 and Cor. 4.1] {Cresson2003}. 
\end{remark} 
These results can be summarized in the following
%%%%%%%%%%%%%%
% Proposition
%%%%%%%%%%%%%
\begin{proposition}
 For all $0<\beta \leq 1$ the discontinuity set \fdiffpm {f}{x}{\beta} is $F_\sigma$ of the first category.
\end{proposition}
\begin{proof}
Since \soc{\pm}{\beta} can be identified with the discontinuity set of \fdiffpm {f}{x}{\beta} and it is as a $F_\sigma$-meager set then we can assert that for all $0<\beta \leq 1$ the discontinuity set is   $F_\sigma$ of the first category.
\end{proof}

   %%%%%%%%%%%%%%%%%%%%%%
   %   Theorem
   %%%%%%%%%%%%%%%%%%%%% 
   \begin{proposition}[Null measure property of monotone functions]
   	\label{lem:LebesgueHolder} 
   	Let \textit{f} be $\beta$-differentiable and monotone  in the interval $[a, b]$ for $\beta <1$. 
   	
   	Then the Lebesgue measure $m \left( \soc{\pm}{\beta} \left( f \right)  \right) =0 $.
   \end{proposition}
   \begin{proof}
   	The proof follows from the   Lebesgue's differentiation Theorem.
   	The H\"older condition fulfills the assumption of the theorem.
   	Therefore,  
   	by \cite{Prodanov2015} we have
   	\[	
   	\fdiffplus {f}{x}{\beta} =\frac{1}{\beta} \llim{\epsilon}{0}{ \epsilon ^{1-\beta}  f ^{\prime}(x + \epsilon)  }  \epnt
   	\]
   	Therefore, if $f^\prime(x)$ is unbounded at $a$ then $ \fdiffplus {f}{a}{\beta} \in \soc{+}{\beta} $.
   	Therefore $m \left( \soc{+}{\beta} \left( f \right)  \right) =0 $.
   	Similar reasoning is applicable in the backward case.
   \end{proof}
  %%%%%%%%%%%%%%%%%%%%%%
   %%%%%%%%%%%%%%%%%%%%%%%%%%%%
   \begin{proposition}[Null measure property of bounded variation functions]
   	\label{lem:DenjoyHolder}. 
   	Let \textit{f} be $\beta$-differentiable and of bounded variation in the interval $[a, b]$ for $\beta <1$. 
   	Then the Lebesgue measure $m \left( \soc{\pm}{\beta} \left( f \right)  \right) =0 $.
   \end{proposition}
   %%%%%%%%%%%%%%%%%%%%%%%%%
   \begin{proof}
   	  The proof follows from the Denjoy-Young-Salts theorem  observing that a function  of bounded variation is differentiable a.e. (for such a function is the difference of two increasing functions). 
   	  Then the conditions of Th. \ref{lem:LebesgueHolder} apply.
   \end{proof}
 %%%%%%%%%%%%%%%%%%%%%%%%%
 
	To summarize, we have established that relaxing the requirement of left and right $\alpha$-velocity equality actually allows for accounting of non-trivial cases, which are described by the set of change  \soc{\pm}{\beta}.
 
   %%%%%%%%%%%%%%%%%
   %   Section
   %%%%%%%%%%%%%%%%%
	\section{Some properties on intervals}
	\label{sec:app1}
	
	%%%%%%%%%%%%%%%%
	%  Section
	%%%%%%%%%%%%%%%%
	\subsection{Intermediate value properties}
	\label{sec:ivp}
 %%%%%%%%%%%%%%%%%%%
 %  Property
 %%%%%%%%%%%%%%%%%%%
	\begin{property}[Weak Intermediate Value]
		\label{def:ivp1}
		A non-decreasing (respectively non-increasing) function $f$  has the  weak intermediate value property on the interval $I=[a,b]$ if for every $y$, such that 
		$ f(a) \leq y \leq f(b) $
	 (respectively $f(b) \leq y \leq f(a) 	$) 
	 there exists $x \in [a,b]$, such that $f(x)=y$. 
	\end{property}	
	 %%%%%%%%%%%%%%%%%%%%%%%%%%
	 % Theorem
	 %%%%%%%%%%%%%%%%%%%%%%%%%% 
	 \begin{theorem}[Fractional Darboux Theorem]\label{th:extDarboux}
	Let $f$ be $\beta$-differentiable on the closed interval $I=[a,b]$. 
	If \fdiffpm {f}{x}{\beta}  is continuous on  $I$ then \fdiffpm {f}{x}{\beta} 
	has the \underline{weak intermediate value} property.
	 \end{theorem}
	 %%%%%%%%%%%%%%%%%%%%%
	 \begin{proof}
	 Since $f$ is $\beta$-differentiable on   $I$ then \textit{f} is continuous on $I$.
	 For $\beta=1$ we have the original Darboux theorem.
	 For $\beta<1$ if \fdiffpm {f}{a}{\beta} = \fdiffpm {f}{b}{\beta} = 0  we have  \fdiffpm {f}{x}{\beta} = 0 $ \forall x \notin \soc{\pm}{\beta} $. Since \soc{\pm}{\beta} is disconnected then there exist a point \textit{x} where  \fdiffpm {f}{x}{\beta} = 0.
	 \end{proof} 
	 It is instructive to remark here that for the fractional case the statement of the theorem is only trivially satisfied.
	 On the other hand, if the inequality defining the intermediate value property is made strict then the theorem is satisfied only for ordinary derivatives, i.e. for $\beta=1$.

	 %%%%%%%%%%%%%%%%%%%
	 %  Property
	 %%%%%%%%%%%%%%%%%%%
	 	\begin{property}[Strong Intermediate Value]
	 		\label{def:ivp2}
	 		A  non-decreasing (respectively non-increasing) function $f$ has the  strong intermediate value property  on the interval $I=[a,b]$ if for every $y$, such that 
	 		$ f(a) < y < f(b) $ 
	 		(respectively $ f(b) < y < f(a) $)  there exists $x \in [a,b]$, such that $f(x)=y$. 
	 	\end{property}
	 %%%%%%%%%%%%%%%%%%%%%%%%%%%%%%
	 
	 %%%%%%%%%%%
	 % Corollary
	 %%%%%%%%%%%%
	 \begin{corollary}
	 Let $f$ be $\beta$-differentiable on the closed interval $I=[a,b]$. 
	 If \fdiffpm {f}{x}{\beta} is continuous on $I$  and has the \underline{strong intermediate value} property then $\beta=1$.
	 \end{corollary}
	 %%%%%%%%%%%%%%%%%%%%%%%%%	
	\begin{proof}
	Since $f$ is $\beta$-differentiable on   $I$ then \textit{f} is continuous on $I$.
	Then the proof follows from the observation that for $\beta<1$  $\fdiffpm {f}{x}{\beta}=0$, which violates the hypothesis of the strong intermediate value property.
	\end{proof}
    %%%%%%%%%%%%%%%%%%%%%%%%%%
    
   	%%%%%%%%%%%%%%%%
   	%  Section
   	%%%%%%%%%%%%%%%%
   	\subsection{Mean value properties}
   	\label{sec:mvp}
	  %%%%%%%%%%%%%%%%%%%%%%%%%%
	  % Theorem
	  %%%%%%%%%%%%%%%%%%%%%%%%%% 	
	  \begin{theorem}[Fractional Rolle]
	  	\label{th:frrolle}
	  	Let %$f(x) \in \holder{r, \beta}$ in $ \left[ a, b\right] $ 
	  	\textit{f} be $\beta$-differentiable on the closed interval $I=[a,b]$
	  	and $f(a) = f(b)$ then
	  	there exists a number $c \in \left[a, b \right]  $ such that 
	  	$ \fdiffplus{f}{c}{\beta} \leq 0 $ and $ \fdiffmin{f}{c}{\beta} \geq 0 \,$.
	  	Respectively,
	  	$\fdiffmin{f}{c}{\beta}\leq 0$ and $\fdiffplus{f}{c}{\beta}\geq 0 \,$.
	  	If both variations agree then
	  	$\fdiffmin{f}{c}{\beta}  =\fdiffplus{f}{c}{\beta}=0 \,$.
	  \end{theorem}
	  %%%%%%%%%%%%%%%%%%%%%%%%%%%
	  \begin{proof}
	  	The proof of the theorem follows closely the proof of the generalized Rolle's Theorem.
	  	We can distinguish two cases:	
	  	\begin{description}
	  		\item[End-interval case] 
	  		
	  		Since $f(x)$ is continuous on $ \left[ a, b\right] $, it follows from the continuity property that $f(x)$ attains a maximum \textit{M} at some $c_1 \in  \left[ a, b\right] $ and a minimum \textit{m} at some $c_2 \in  \left[ a, b\right] $.
	  		
	  		Suppose $c_1$ and $c_2$ are both endpoints of $ \left[ a, b\right] $.
	  		Because $f(a)=f(b)$ it follows that $m=M$ and so $f(x)$ is constant on $ \left[ a, b\right] $.
	  		Hence $ \fracvarplus{f}{x}{\beta} = \fracvarmin{f}{x}{\beta} =0, \ \forall x \in \left[ a, b\right]$.
	  		% % % % % % % % % % % % % % % % % % % % %
	  		\item[Interior case]
	  		Suppose then that the maximum is obtained at an interior point \textit{c} of $\left( a,b \right)$.
	  		For an $\epsilon>0\,$, such that $c + \epsilon \in \left[ a, b\right]$  $f(c+\epsilon) \leq f(c)$ by assumption. 
	  		Therefore, $ \fracvarplus{f}{c}{\beta} \leq 0$. 
	  		Therefore, $\fdiffplus{f}{c}{\beta} \leq 0 \,$.
	  		Similarly, for $c - \epsilon \in \left[ a, b\right]$ we have $\deltamin{f}{c} \geq 0 $ and
	  		$ \fracvarmin{f}{c}{\beta} \geq 0$.
	  		Therefore, $\fdiffmin{f}{c}{\beta} \geq 0 \,$.
	  		Finally, if $\fdiffmin{f}{c}{\beta} =\fdiffplus{f}{c}{\beta}$ then both equal $0$.
	  		The proof for the minimum follows identical reasoning. 
	  	\end{description}
	  \end{proof}
	  %%%%%%%%%%%%%%%%%%%%%%%%%%%%%%%%%%%%%%%%%%%%%%%%
	  An analog of the Mean Value Theorem for fractional orders can also be formulated. 
	  
	  %%%%%%%%%%%%%%%%%
	  %  Property
	  %%%%%%%%%%%%%%%%%
	  	\begin{property}[Weak Mean Value]
	  		\label{def:mvp1}
	  		The fractional velocity \fdiffpm{f}{x}{\beta} has the \underline{weak mean value} property on the closed interval $I=[a,b]$  if
	  		\fdiffpm{f}{x}{\beta} is continuous on the open interval $(a,b)$ and
	  		for some $c \in I$     
	  		\[
	  		\fdiffpm{f}{c}{\beta} =\frac{f(b) -f(a)}{\left( b-a\right)^\beta } \neq 0 \epnt  
	  		\]
	  	\end{property}

	  %%%%%%%%%%%%%%%%%%%%%%%%%%
	  % Theorem
	  %%%%%%%%%%%%%%%%%%%%%%%%%% 
	  \begin{theorem}[Fractional Mean Value]	
	  	Let	\textit{f} be $\beta$-differentiable of order $\beta< 1$  on the closed interval $I=[a,b]$ and $f(a) \neq f(b)$.
	  	If 
	  	\fdiffpm{f}{x}{\beta} has the weak mean value property on \textit{I} then $c =a $ or $c=b$.
	  \end{theorem}
	  %%%%%%%%%%%%%%%%%%%%%%%%
	  \begin{proof}
	  	If $f(a) = f(b)$ then by Th. \ref{th:fdiffcont} we have an infinite number of points for which \fdiffpm{f}{c}{\beta}=0.
	  	Therefore, let's assume that $f(a) \neq f(b)$.
	  	
	  	 Define $g(x)=f(x)-r \left(x-a \right) ^\beta$ , where  \textit{r} is a constant.
	  	 Since \textit{f} is $\beta$-differentiable on  [a,b], the same is true for  \textit{g} as $r \left(x-a \right) ^\beta$ is $\beta$-differentiable. 
	  	 We will select \textit{r} so that  \textit{g} satisfies the conditions of  Fractional Rolle's theorem (Th. \ref{th:frrolle}). 
	  	  \begin{flalign*}g(a)=g(b)&\iff f(a)-r (a - a) ^\beta=f(b)-r (b-a)^\beta 
	  	   \iff r=\frac{f(b)-f(a)}{\left( b -a \right) ^\beta}\cdot
	  	  \end{flalign*} 
	  	 By Th. \ref{th:frrolle}, since \textit{g} is $\beta$-differentiable and $g(a)=g(b)$, 
	  	 there is some  $ c \in [a,b]$ for which  \fdiffplus{g}{c}{\beta}=0, 
	  	 and it follows from the equality  $g(x)=f(x)-r \left( x -a \right) ^\beta$ that
	  	 \[
	  	 \fdiffplus{f}{c}{\beta}= \fdiffplus{g}{c}{\beta}  + r =  0+r= \frac{f(b)-f(a)}{\left( b -a \right) ^\beta}
	  	 \]
	  	 for $x=a$. 
		 By the continuity assumption if $c$ is an interior point then $\fdiffplus{f}{c}{\beta}=0$, therefore $c=a$. 
		 The  case for $c=b$ follows by considering $g(x)=f(x)-r \left( b-x \right) ^\beta$ and applying properties of 
		 \fdiffmin{f}{x}{\beta}.
	  \end{proof}
	  %%%%%%%%%%%%%%%%%%%%%%%%%%
	  The arguments in the proof also establish a stronger statement for $\beta$.
	  It is instructive to demonstrate that if the interval defining the mean value property is left open then the theorem is satisfied only for ordinary derivatives, i.e. for $\beta=1$.
	    %%%%%%%%%%%%%%%%%
	    %  Property
	    %%%%%%%%%%%%%%%%%
	    \begin{property}[Strong Mean Value]
	    	\label{def:mvp2}
	    	The fractional velocity \fdiffpm{f}{x}{\beta} has the  strong mean value property on the open interval $I=(a,b)$  if
	    	\fdiffpm{f}{x}{\beta} is continuous on the open interval $(a,b)$ and
	    	for some $c \in I$     
	    	\[
	    	\fdiffpm{f}{c}{\beta} =\frac{f(b) -f(a)}{\left( b-a\right)^\beta } \neq 0 \epnt  
	    	\]
	    \end{property}	
	    %%%%%%%%%%%%%%%%%%%%%%%%%%%%%%
	  \begin{corollary}
	  	Let	\textit{f} be $\beta$-differentiable on the open interval $I=(a,b)$ and $f(a) \neq f(b)$.
	  	If \fdiffpm{f}{x}{\beta} has the \underline{strong mean value} property on \textit{I} then $\beta=1$.
	  \end{corollary}	
	  %%%%%%%%%%%%%%%%%%%%
	  \begin{proof}
	   The proof follows from the argument just established from the observation that 
	   if  \textit{c} is an interior point in  \textit{I} then $\fdiffplus{f}{c}{\beta}=0$ for $\beta<1$.
	   Therefore, $\beta=1$.
	  \end{proof}
	  %%%%%%%%%%%%%%%%%%%%%%

      %%%%%%%%%%%%%%%%%%%%%%%%%%%%
      % \include{mixedderiv} 
      
      %%%%%%%%%%%%%%%
      %  Section
      %%%%%%%%%%%
      \section{Relation to the integral-based local fractional derivatives}
      \label{sec:KG}
      
      %\input{kgequiv} 
   
   %%%%%%%%%%%%%%%
   %  Section
   %%%%%%%%%%%
   \subsection{Fractional integrals and derivatives}
   \label{sec:fi}
   The left Riemann-Liouville differ-integral of order $\beta \geq 0$ is defined as
   \[
   \frdiffiix{\beta}{ a+ }  f (x) = 
   \dfrac{1}{\Gamma(\beta)} \int_{a}^{x}   f \left( t \right)  \left( x-t \right)^{\beta -1}dt 
   \]
   while the right integral is defined as
   \[
   \frdiffiix{\beta}{ -a }  f (x) = 
   \dfrac{1}{\Gamma(\beta)} \int_{x}^{a}   f \left( t \right)  \left( t-x \right)^{\beta -1}dt 
   \]
   where $\Gamma(x) $ is the Euler's Gamma function  (Samko et al. \cite{Samko1993} [p. 33]). 
   The left (resp. right) Riemann-Liouville fractional derivatives are defined as the expressions (Samko et al. \cite{Samko1993}[p. 35]):
   \begin{flalign*}
   \mathcal{D}_{a+}^{\beta} f (x)  & := \frac{d}{dx} \frdiffiix{1-\beta}{ a+ }  f (x)  = \frac{1}{\Gamma(1- \beta)}  \frac{d}{dx}  \int_{a}^{x}\frac{  f (t ) }{{\left( x-t\right) }^{\beta }}dt  \\
   \mathcal{D}_{-a}^{\beta} f (x) & := -\frac{d}{dx}  \frdiffiix{1-\beta}{ -a }  f (x) = -\frac{1}{\Gamma(1- \beta)} \frac{d}{dx}  \int_{x}^{a}\frac{  f (t )  }{{\left( t-x\right) }^{\beta }}dt
   \end{flalign*}
   
   The left (resp. right) Riemann-Liouville derivative of a function $f$ exists for functions  represented by the left (resp. right) fractional integrals of order $\alpha$ of a Lebesgue-integrable function (Samko et al. \cite{Samko1993}[Definition 2.3, p. 43]). 
   That is for members of the functional spaces
   \begin{flalign*}
   \mathcal{I}^{\alpha}_{a, + (-)} (L^1) &:= \left\lbrace f: 
   \frdiffiix{ \alpha}{ a+ (-a) }  f (x) \in AC([a, b]), f  \in L^1 ([a,b]) , x \in [a, b] \right\rbrace  \ecma 
   %\mathcal{I}^{\alpha}_{a, -} (L^1)& := \left\lbrace f: 
   %\frdiffiix{ \alpha}{ -a }  f (x) \in AC([a, b]), f  \in L^1 ([a,b]) , x \in [a, b] \right\rbrace 
   \end{flalign*}
   respectively. 
   Here $AC$ denotes absolute continuity on an interval in the conventional sense. 
   Samko et al. comment that the existence of a summable derivative $f^\prime(x)$ of a function $f( x)$ does not yet guarantee
   the restoration of $f(x)$ by the primitive in the sense of integration and go on to give references to singular functions for which
   the derivative vanishes almost everywhere and yet the function is not constant, such as for example, the De Rhams's function \cite{Rham1957}. 
   For that purpose, based on Th. 2.3 they introduce another space of 
   \textit{summable fractional derivatives}, for which the inversion property holds. 
   That is  
   \[
   \mathcal{D}_{a+}^{\beta}\ \frdiffiix{ \alpha}{ a+ } \, f = f (x) 
   \]
   for $f \in  \mathcal{I}^{\alpha}_{a, +} (L^1)$ while 
   \[
   \frdiffiix{ \alpha}{ a+ } \  \mathcal{D}_{a+}^{\beta} \, f = f (x) - f(a)
   \]
   for $f \in E^{\alpha}_{a, +} ([a, b])$ (Samko et al. \cite{Samko1993}[Th. 4, p. 44]).
   This space $E^{\alpha}_{a, +} ([a, b])$ can be defined as
   %%%%%%%%%%%%%%%%%%%%%%%%%%
   % Definition
   %%%%%%%%%%%%%%%%%%%%%%%%%% 
   \begin{definition}
   	\label{def:summdfderiv}
   	Define the spaces of  summable fractional derivatives (Samko et al. \cite{Samko1993}[Definition 2.4, p. 44])  as
	$
	E^{\alpha}_{a,\pm} ([a, b]):= \left\lbrace f: 
	\mathcal{I}^{1-\alpha}_{a, \pm} (L^1)
	\right\rbrace 
	$.
   \end{definition}
   %%%%%%%%%%%%%%%%%%%%%%%
   
   %%%%%%%%%%%%%%%
   %  Section
   %%%%%%%%%%%
   \subsection{Local(ized) fractional derivative}
   \label{sec:lfd}
   
   The definition of \textsl{local fractional derivative} (LFD)  introduced by Kolwankar and Gangal \cite{Kolwankar1997} is based on the localization of Riemann-Liouville fractional derivatives towards a particular point of interest in a way similar to Caputo.
   The seminal publication defined only the left derivative. 
   %%%%%%%%%%%%%%%
   %  Definition
   %%%%%%%%%%%%%%%
   \begin{definition}
   	\label{def:kg-lfd}
   	Define left LFD as
   	\[ 
   	\mathcal{D}_{KG+}^{\beta}  f(x) := \llim{x}{a}{} \mathcal{D}_{a+}^{\beta} \left(  f-f(a) \right) (x) % = \dfrac{1}{\Gamma(1- \beta)} \llim{x}{a}{} \frac{d}{dx}  \int_{a}^{x}\frac{  f (t ) -f(a) }{{\left( x-t\right) }^{\beta }}dt    \epnt
   	\]
   	and right LFD as
   	\[ 
   	\mathcal{D}_{KG-}^{\beta}  f(x) :=  \llim{x}{a}{} \mathcal{D}_{-a}^{\beta} \left(  f(a) - f \right) (x) \epnt % = -\dfrac{1}{\Gamma(1- \beta)} \llim{x}{a}{} \frac{d}{dx}  \int_{x}^{a}\frac{  f (t ) -f(a) }{{\left( t-x\right) }^{\beta }}dt    \epnt
   	\] 
   \end{definition}
   %%%%%%%%%%%%%%%%%%
   
   Chen et al. \cite{Chen2010} and Ben Adda and Cresson \cite{Adda2001} indicated that the Kolwankar -- Gangal  definition of local fractional derivative is equivalent to Cherbit's definition for certain classes of functions.
   On the other hand, some inaccuracies have been observed in these articles \cite{Chen2010, Adda2013}. 
   Since the results of Chen et al. \cite{Chen2010} and Ben Adda-Cresson  \cite{ Adda2013} are proven under different hypotheses and notations I feel that a separate proof of the equivalence results using the theory established so-far is in order.
   
   %%%%%%%%%%%%%%%%%%
   %  Prop
   %%%%%%%%%%%%%%%%%
   \begin{proposition}[LFD equivalence]
   	\label{prop:ldfeq}
   	%Let $f(x) \in \holder{r,\beta}$ about $x$. 
   	Let $f(x)$ be $\beta$-differentiable about $x$. Then $\mathcal{D}_{KG, \pm }^{\beta} f(x)$ exists and
   	\[
   	\mathcal{D}_{KG, \pm }^{\beta} f(x)  = \Gamma(1+\beta) \ \fdiffpm {f}{x}{\beta} \epnt
   	\]
   \end{proposition}
   %%%%%%%%%%%%%%%%%%%
   %%%%%%%%%%%%%%%%%%%%
   \begin{proof}
   	We will assume that $f(x)$ belongs to \holder{r,\beta} in the interval $ [a, a+x]$.
   	Since $x$ will be treated as a variable, for simplicity let's assume that $\fdiffplus {f}{a}{\beta} \in \soc{}{\beta}$.
   	Then by Cor. \ref{th:holcomp1} we have
   	\[
   	f(z ) = f(a) + \fdiffplus {f}{a}{\beta} (z - a)^\beta + \bigoh{(z - a)^\beta}, \ a \leq z \leq x \epnt
   	\]
   	Standard treatments of the fractional derivatives \cite{Oldham1974} and the changes of variables $u=(t-a)/(x-a)$ give the alternative  formulation
   	\[
   	\mathcal{D}_{+a}^{\beta} f(x) = \frac{\partial}{\partial h} \left( \frac{ h^{1-\beta}}{\Gamma(1- \beta) }  \ \int\limits_0^1 \frac{ f( h u +a ) - f(a)}{(1-u)^{\beta} } du \right) 
   	\]
   	for $ h=x-a$.
   	Therefore, we can evaluate the fractional  Riemann-Liouvulle integral as follows:
   	\[
   	\frac{ h^{1-\beta} }{\Gamma(1- \beta) }    \int\limits_0^1 \frac{ f( h u +a ) - f(a)}{(1-u)^{\beta} } du= 
   	\frac{ h^{1-\beta} }{\Gamma(1- \beta) }    \int\limits_0^1 \frac{ K  \left(  h  u\right) ^\beta  + \bigoh{(h u)^\beta}  }{(1-u)^{\beta} } du = I
   	\]
   	setting conveniently $K= \fdiffplus {f}{a}{\beta} $. 
   	The last expression can be evaluated in parts as %follows:
   	\[
   	I =  \underbrace{\dfrac{ h^{1-\beta} }{\Gamma(1- \beta) }    \int\limits_0^1 \frac{ K h^\beta u^\beta   }{(1-u)^{\beta} } du }_A +
   	\underbrace{\dfrac{ h^{1-\beta} }{\Gamma(1- \beta) }    \int\limits_0^1 \frac{   \bigoh{(h u)^\beta}  }{(1-u)^{\beta} } du}_C  \epnt
   	\]
   	The first expression is recognized as the Beta integral \cite{Oldham1974}:
   	\[
   	A = \frac{ h^{1-\beta} }{\Gamma(1- \beta) } B \left(1-\beta, 1+\beta \right) h^\beta  K = \Gamma(1+\beta) \, K h
   	\] 
   	In  order to evaluate the second expression we observe that 
   	by  Lemma \ref{th:bondvar1}  
   	\[\left|  \bigoh{(h u)^\beta} \right|   \leq \gamma (h u)^\beta \]  
   	for a  positive $\gamma = \bigoh{ 1} $. 
   	Assuming without loss of generality that $f(x)$ is non decreasing in the interval we have
   	$
   	C \leq \Gamma(1+\beta) \, \gamma h
   	$ and 
   	\[
   	\mathcal{D}_{a+}^{\beta} f(x) \leq \left( K + \gamma\right) \Gamma(1+\beta)
   	\]
   	and the limit gives
   	$\llim{x}{a+}{K + \gamma} = K$ by the \textit{squeeze lemma} and Cor. \ref{prop:bigoh}.
   	Therefore, 
   	$
   	\mathcal{D}_{KG + }^{\beta} f(a)  = \Gamma(1+\beta) \fdiffplus {f}{a}{\beta} 
   	$.
   	On the other hand, for   \holder{r,\alpha} and $\alpha > \beta$ by the same reasoning 
   	\[
   	A = \frac{ h^{1-\beta} }{\Gamma(1- \beta) } B \left(1-\beta, 1+\alpha \right) h^\alpha  K = \Gamma(1+\beta) \, K h^{1- \beta+ \alpha } \epnt
   	\]
   	Then differentiation by \textit{h} gives
   	\[ 
   	A^\prime_h=  \frac{\Gamma(1+\alpha)}{\Gamma(1+\alpha -\beta)} \, K h^{ \alpha - \beta } \epnt
   	\]
   	Therefore,
   	\[
   	\mathcal{D}_{KG+}^{\beta} f(x) \leq \frac{\Gamma(1+\alpha)}{\Gamma(1+\alpha -\beta)} \left( K + \gamma\right)   h^{ \alpha - \beta } \epnt
   	\] 
   	Therefore, 
   	$
   	\mathcal{D}_{KG \pm }^{\beta} f(a)  = \fdiffpm {f}{a}{\beta} = 0 
   	$.
   	Finally, for $\alpha =1$ the expression $A$ should be evaluated as the limit $\alpha \rightarrow 1$ due to divergence of the $\Gamma$ function.
   	The proof for the left LDF follows identical reasoning, observing the backward fractional Taylor expansion property.   	
   \end{proof}
   %%%%%%%%%%%%%%%%%%
   
   The weaker condition of only point-wise \holder{\beta} continuity would require the additional hypothesis of summability as identified in \cite{Adda2013}. The following results can be stated.
   
   %%%%%%%%%%%%%%%%%%
   %  Prop
   %%%%%%%%%%%%%%%%%
   \begin{proposition}
   	\label{prop:KGconverse1}
   	Suppose that $\mathcal{D}_{KG \pm}^{\beta}  f(x) $ exists finitely implying only $f \in  L^1  $ in the interval $[a, x +\epsilon]$.
   	Then condition \ref{C1} holds for $f$ in this interval.
   \end{proposition}
   %%%%%%%%%%%%%%%%%%
   \begin{proof}
   	  	
   	The left R-L derivative can be evaluated as follows. 
   	Consider the fractional integral in the Liouville form
   	\begin{flalign*}
   	I_1 &=   \int\limits_{0}^{\epsilon+ x-a} \frac{ f(x+ \epsilon - h) - f(a)}{h^\beta}\, d h -
   	\int\limits_{0}^{  x-a} \frac{ f(x - h) - f(a)}{h^\beta}\, d h  \\
   	& =  \underbrace{  \int\limits_{x-a}^{\epsilon+ x-a} \frac{ f(x+ \epsilon - h) - f(a)}{h^\beta}\, d h}_{I_2}
   	+ \underbrace{  \int\limits_{0}^{x-a} \frac{ f(x +\epsilon - h) - f(x-h)}{h^\beta}\, d h}_{I_3}
   	\end{flalign*}
   	Without loss of generality assume that $f$ is non-decreasing in the interval $[a, x + \epsilon - a]$ and 
   	set
   	$ M_{y, x} = \sup_{[x, y]} f  - f(x)$ and $  m_{y, x} = \inf_{[x, y]} f - f(x)$.
   	Then  
   	\[
   	I_2 \leq  \int\limits_{x-a}^{\epsilon+ x-a} \frac{ M_{x+ \epsilon , a} }{h^\beta}\, d h = 
   	\frac{M_{x+ \epsilon , a}}{1 -\beta}  \left(  { {\left( x- \epsilon + a\right) }^{1-\beta} -\left( x-a\right) }^{1-\beta} \right) \leq   \epsilon \frac{M_{x+ \epsilon , a}  }{\left( x-a\right)^\beta } + \bigoh{\epsilon^2}
   	\]
   	for $x \neq a$.
   	In a similar manner
   	$$
   	I_2  \geq m_{x+ \epsilon , a} \frac{\epsilon }{\left( x-a\right)^\beta } + \bigoh{\epsilon^2} \epnt
   	$$ 
   	Then dividing by $\epsilon$ gives
   	\[
   	\frac{  m_{x+ \epsilon , a}}{\left( x-a\right)^\beta } + \bigoh{\epsilon } \leq \frac{I_2}{\epsilon} \leq   \frac{  M_{x+ \epsilon , a}}{\left( x-a\right)^\beta } + \bigoh{\epsilon } 
   	\]
   	Therefore, the quotient limit is bounded from both sides as
   	\[
   	\frac{ m_{x , a}}{ \left( x-a\right)^\beta }\leq   \underbrace{ \llim{\epsilon}{0}{\frac{I_2}{\epsilon}} }_{I_2^\prime} \leq \frac{  M_{x , a}}{ \left( x-a\right)^\beta }
   	\]
   	by the continuity of $f$. 
   	In a similar way we establish
   	\begin{flalign*}
   	I_3  \leq  \int\limits_{0}^{x-a} \frac{ M_{x +\epsilon , x} }{h^\beta} \, d h  = \frac{  M_{x +\epsilon , x} }{1-\beta}\left( x-a \right)^{1-\beta} 
   	\end{flalign*}
   	and
   	\[
   	\frac{  m_{x +\epsilon , x} }{1-\beta}\left( x-a \right)^{1-\beta} \leq I_3 
   	\]
   	Therefore,
   	\[
   	\frac{  m_{x +\epsilon , x} }{\left( 1-\beta \right) \epsilon }\left( x-a \right)^{1-\beta} \leq \frac{I_3}{\epsilon} \leq 
   	\frac{  M_{x +\epsilon , x} }{\left( 1-\beta \right) \epsilon }\left( x-a \right)^{1-\beta} 
   	\]
   	By the absolute continuity of the integral the quotient limit
   	$
   	\frac{I_3}{\epsilon} 
   	$
   	 exists as $\epsilon \rightarrow 0$ for almost all $x$ and  is bounded between
   	\[
   	m^{\star}_{x +\epsilon , x} \frac{\left( x-a \right)^{1-\beta} }{\left( 1-\beta \right)  } \leq \underbrace{ \llim{\epsilon}{0}{\frac{I_3}{\epsilon}} }_{I_3^\prime} \leq   M^{\star}_{x +\epsilon , x} \frac{\left( x-a \right)^{1-\beta} }{\left( 1-\beta \right)  }  
   	\]
   	where 
   	$ M^{\star}_{x +\epsilon, x} = \sup_{[x, x +\epsilon]} f^\prime  - f^\prime(x)$ and $  m^{\star}_{x +\epsilon, x} = \inf_{[x, x +\epsilon]} f^\prime  - f^\prime(x)$ wherever these exist.
   	Therefore, as $x$ approaches $a$ 
   	$
   	\llim{x}{a}{I_3^\prime}=0
   	$.
   	
   	Finally, we establish the bounds of the limit
   	\[
   	0 \leq \llim{x}{a}{}\frac{  m_{x , a}}{ \left( x-a\right)^\beta } \leq \llim{x}{a}{I_2^\prime} \leq \llim{x}{a}{}\frac{  M_{x , a}}{ \left( x-a\right)^\beta } \epnt
   	\]
   	Therefore, condition \ref{C1} is necessary for the existing of the limit and hence for
   	$\llim{x}{a}{I^\prime}$ .
 
   \end{proof}
   %%%%%%%%%%%%%%%%%%
   However, since Condition \ref{C1} is only necessary for the existence of fractional velocity then
   $f$ may not  be $\beta$-differentiable at $x$.
   
   %%%%%%%%%%%%%%%%%%
   %  Prop
   %%%%%%%%%%%%%%%%%
   \begin{proposition}
   	\label{prop:KGconverse2}
   	Suppose that $\mathcal{D}_{KG \pm}^{\beta}  f(x) $ exists finitely and the related R-L derivative is summable in the sense of Def. \ref{def:summdfderiv}.
   	Then   $f$ is $\beta$-differentiable about $x$
   	and $ 	\mathcal{D}_{KG, \pm }^{\beta} f(x)  = \Gamma(1+\beta) \ \fdiffpm {f}{x}{\beta} 
   	$.
   \end{proposition}
   %%%%%%%%%%%%%%%%%%
   \begin{proof}
   	Suppose that $f \in  E^{\alpha}_{a,+} ([a, a + \delta])$ and let $\mathcal{D}_{KG +}^{\beta}  f(x)=L $. 
   	The existence of the limits supposes that
   	\[
   	\left|  \mathcal{D}_{a+}^{\beta} \left(  f-f(a) \right) (x)  - L \right| < \mu  
   	\] for $|x-a|  \leq  \delta$ and a Cauchy sequence $\mu$. 
   	
   	Without loss of generality suppose that $\mathcal{D}_{a+}^{\beta} \left(  f-f(a) \right) (x) $ is non-decreasing and $L \neq 0$.
   	We integrate the inequality:
   	\[
   	\frdiffiix{ \alpha}{ a+ } \ \left( \mathcal{D}_{a+}^{\beta} \left(  f-f(a) \right) (x)  - L  \right)  < \frdiffiix{ \alpha}{ a+ } \mu 
   	\]
   	Then by the inversion property 
   	\[
   	f (x) - f(a) -  \frac{ L }{\Gamma(1+\beta)} (x-a)^\beta < \frac{ \mu (x-a)^\beta}{\Gamma(1+\beta)}  
   	\]
   	and
   	$$
   	\frac{ 	f (x) - f(a) - L/\Gamma(1+\beta) }{ (x-a)^\beta} < \frac{\mu}{\Gamma(1+\beta)} = \bigoh {1} \epnt
   	$$
   	Therefore, by  Corr. \ref{prop:bigoh}  $f$ is $\beta$-differentiable and
   	$\mathcal{D}_{KG, + }^{\beta} f(x)  = \Gamma(1+\beta) \ \fdiffplus{f}{x}{\beta}$.
   	The last assertion comes from Prop. \ref{prop:ldfeq}.
   	The right case can be proven in a similar manner. 
   \end{proof}
   %%%%%%%%%%%%%%%%%%%%

   %%%%%%%%%%%%%%
   %  Seection
   %%%%%%%%%%%%%%
   \section{Discussion} 	
	\label{sec:disc}

   %%%%%%%%%%%%
   %  Section
   %%%%%%%%%%%%%
   \subsection{Overall utility}
   \label{sec:overall}
   As demonstrated here, fractional velocities can be used to characterize the set of discontinuities of  derivatives of H\"olderian functions. 
   Presented results further demonstrate that regular H\"older functions can be approximated locally as fractional powers of appropriate order.
   Moreover, in a direct extension of this application fractional velocities can be used to provide fractional Taylor expansions 
   and to regularize derivatives of H\"older functions at non--differentiable points \cite{Prodanov2016a}.    
   %Finally, the fractional velocities provide a tool to determine the point-wise H\"older exponents.
   	
	%%%%%%%%%%%%%%
	%  Section
	%%%%%%%%%%%%%%
    \subsection{Link to integral local fractional derivatives}    
    \label{sec:link}
    Kolwankar-Gangal local fractional derivative has been introduced to study scaling of physical systems and  systems exhibiting fractal behavior \cite{Kolwankar1998}.  
  	The conditions for applicability of the Kolwankar-Gangal fractional derivative were not specified in the seminal paper, which has left space for different interpretations. 
  	For example, the initial claim in \cite{Adda2001} that  local  derivative is equivalent to the limit of a difference quotient needed to be clarified in \cite{Chen2010, Adda2013} and restricted to the more limited functional space of summable fractional Riemann-Liouville derivatives. 
  	In addition, the domain of existence of the Kolwankar-Gangal  derivative require certain regularity conditions on the image function. 	
 	%This all limits the use of this operator in applications. 
	On the other hand, the overlap of the definitions of the Cherebit's fractional velocity and the Kolwankar-Gangal fractional derivative is not complete.
	Notably, Kolwankar-Gangal fractional derivatives are sensitive to the critical local H\"older exponents, while the fractional velocities are sensitive to the critical point-wise H\"older exponents and there is no complete equivalence between those quantities \cite{Kolwankar2001}.

   Presented results call for a careful inspection of the claims branded under the name of "local fractional calculus". 
   Specifically, implied conditions on image function's regularity and the continuity of resulting \textsl{local fractional derivative} must be examined in all cases. 
   For example, the numerical schemes involving fractional variation operators (in our notation) with the aim of approximating fractional velocity actually make the silent assumption that the points are sampled from the set of change $\chi^{\beta}$, which must be separately verified in each case.

    %%%%%%%%%%%%
    %  Section
    %%%%%%%%%%%%%  	
   \subsection{Physical applications}
   	Classical physical variables, such as velocity or acceleration, are considered to be differentiable functions of position. On the other hand, typical quantum mechanical paths \cite{Abbott1981, Amir-Azizi1987, Sornette1990} and   Brownian motion trajectories were found to be typically non-differentiable.
   	The relaxation of the differentiability assumption opens new avenues in describing physical phenomena \cite{Nelson1966, Nottale1989} but also challenges existing mathematical methods.
   	H\"olderian functions in this regard can be used as building blocks of such strongly non-linear models.
    In contrast to usual fractional derivative, the physical interpretation of this quantity is easier to establish due to its local character and the demonstrated fractional Taylor-Lagrange property.

   %%%%%%%%%%%%
   %  Section
   %%%%%%%%%%%%%     
\section*{Acknowledgments}
The work has been supported in part by a grant from Research Fund - Flanders (FWO), contract numbers G.0C75.13N, VS.097.16N.
The author would like to acknowledge Dr. Dave L. Renfro and 
Dr. Jacques L\'{e}vy V\'{e}hel for helpful feedback on the manuscript.

%%%%%%%%%%%%%%55
% Section
%%%%%%%%%%%%%%%%%
\bibliographystyle{plain}  % Style BST file 
\bibliography{qvar1}

\end{document}

%% file: command11.tex
	\newcommand{\fdiff}[4]{\ensuremath{ \upsilon^{#4 }_{ #3 }  #1  \left(  #2 \right)  }}
	\newcommand{\fdiffplus}[3]{ \fdiff {#1}{#2}{+}{#3} }
	\newcommand{\fdiffmin}[3]{ \fdiff {#1}{#2}{-}{#3} }
	\newcommand{\fdiffpm}[3]{ \fdiff {#1}{#2}{\pm}{#3} }
	 \newcommand{\fracvar}[4]{\ensuremath{ \upsilon_{ #3 }^{ #4} \left[ #1 \right] \left(  #2 \right)   }}
	 \newcommand{\fracvarplus}[3]{ \fracvar {#1}{#2}{#3}{\epsilon+} }
	 \newcommand{\fracvarmin}[3]{ \fracvar {#1}{#2}{#3}{\epsilon -} }
	  \newcommand{\fracvarpm}[3]{ \fracvar {#1}{#2}{#3}{\epsilon \pm} }
	 
	\newcommand{\llim}[3]{\ensuremath{ \lim\limits_{ #1 \rightarrow #2} #3 }}
	\newcommand{\fclass}[2]{\ensuremath{  \mathbb{#1}^{\, #2} }}

	\newcommand{\soc}[2]{\ensuremath{ \chi_{#1}^{#2} }}
	
	\newcommand{\holder}[1]{\fclass{H}{#1} }
	
	 \newcommand{\osc}[4]{\ensuremath{ \mathrm{osc}_{ #3 }^{ #4} [ #1 ] \left(  #2 \right)   }}

	\newcommand\smallO{
		\mathchoice
		{{\scriptstyle\mathcal{O}}}% \displaystyle
		{{\scriptstyle\mathcal{O}}}% \textstyle
		{{\scriptscriptstyle\mathcal{O}}}% \scriptstyle
		{\scalebox{.7}{$\scriptscriptstyle\mathcal{O}$}}%\scriptscriptstyle
	}
	
	\newcommand{\bigoh}[1]{ \ensuremath{ \smallO \left(  #1 \right) }   }
	
	\newcommand{\deltaop}[4]{\ensuremath{ \Delta_{ #3 }^{ #4} \left[ #1 \right] \left(  #2 \right)   }}
	\newcommand{\deltaplus}[2]{ \deltaop {#1}{#2}{\epsilon}{+} }
	\newcommand{\deltamin}[2]{ \deltaop {#1}{#2}{\epsilon}{-} }
	
	\newcommand{\sep}{,}
	\newcommand{\epnt}{\; .}
	\newcommand{\ecma}{\; ,}

\newcommand{\frdiffii}[3]{\ensuremath{\,_{#3}\mathbf{I}^{#1}_{#2}}}

\newcommand{\frdiffiix}[2]{\frdiffii{#1}{x}{#2}}

 \newtheorem{theorem}{Theorem}
 \newtheorem{lemma}{Lemma}
 \newtheorem{corollary}{Corollary}
 \newtheorem{proposition}{Proposition}
 
  \newtheorem{condition}{Condition}
 \newtheorem{property}{Property}

\newtheorem{definition}{Definition}
 \newtheorem{remark}{Remark}
 \newtheorem{example}{Example}